\documentclass[leqno,11pt]{article}
\usepackage[utf8]{inputenc}
\usepackage[T1]{fontenc}
\usepackage{microtype}

\usepackage[letterpaper]{geometry}

\ifdefined\screenview
  \edef\mtht{\the\textheight}
  \edef\mtwd{\the\textwidth}
  \usepackage{xcolor}
  \definecolor{BackgroundColor}{RGB}{253, 246, 227}
  \pagecolor{BackgroundColor}
\fi

\usepackage{amsmath}
\usepackage{amsthm}
\usepackage{tikz-cd}
\usetikzlibrary{arrows} 
\tikzset{
  commutative diagrams/.cd, 
  arrow style=tikz, 
  diagrams={>=stealth}
}
\usetikzlibrary{matrix,decorations.pathreplacing,calc}

\usepackage{textcomp}

\usepackage[sb]{libertine}
\usepackage[varqu,varl]{zi4}
\usepackage[libertine,bigdelims,vvarbb]{newtxmath} 
\usepackage[supstfm=libertinesups,%
  supscaled=1.2,%
  raised=-.13em]{superiors}
\useosf 
\usepackage[scr=boondox,cal=euler]{mathalfa}

\usepackage{slashed}
\usepackage{esint} 
\usepackage[english]{babel}

\usepackage{imakeidx}
\makeindex[intoc]

\usepackage{csquotes}
\usepackage[
  backend=biber,
  hyperref=true,
  backref=true,
  isbn=false,
  doi=true,
  natbib=true,
  eprint=true,
  useprefix=true,
  maxcitenames=99,
  maxbibnames=99,  
  maxalphanames=99, 
  minalphanames=99,
  safeinputenc,
  style=alphabetic,
  citestyle=alphabetic,
  block=space,
  datamodel=preamble/ext-eprint
]{biblatex}
\usepackage[
  hypertexnames = false,
  colorlinks    = true,
  citecolor     = gray,
  linkcolor     = gray,
  urlcolor      = gray,
  breaklinks
]{hyperref}
\makeatletter
\DeclareFieldFormat{arxiv}{%
  arXiv\addcolon\space
  \ifhyperref
    {\href{http://arxiv.org/\abx@arxivpath/#1}{%
       \nolinkurl{#1}%
       \iffieldundef{arxivclass}
         {}
         {\addspace\texttt{\mkbibbrackets{\thefield{arxivclass}}}}}}
    {\nolinkurl{#1}
     \iffieldundef{arxivclass}
       {}
       {\addspace\texttt{\mkbibbrackets{\thefield{arxivclass}}}}}}
\makeatother
\DeclareFieldFormat{mr}{%
  MR\addcolon\space
  \ifhyperref
    {\href{http://www.ams.org/mathscinet-getitem?mr=MR#1}{\nolinkurl{#1}}}
    {\nolinkurl{#1}}}
\DeclareFieldFormat{zbl}{%
  Zbl\addcolon\space
  \ifhyperref
    {\href{http://zbmath.org/?q=an:#1}{\nolinkurl{#1}}}
    {\nolinkurl{#1}}}
\renewbibmacro*{eprint}{%
  \printfield{arxiv}%
  \newunit\newblock
  \printfield{mr}%
  \newunit\newblock
  \printfield{zbl}%
  \newunit\newblock
  \iffieldundef{eprinttype}
    {\printfield{eprint}}
    {\printfield[eprint:\strfield{eprinttype}]{eprint}}
}
\AtEveryBibitem{%
  \clearlist{publisher}%
}
\AtEveryBibitem{%
  \clearlist{address}%
}
\DeclareFieldFormat[article,inproceedings,inbook,incollection,thesis]{title}{\textit{#1}}
\renewbibmacro{in:}{}
\newcommand{\printreferences}{\printbibliography[heading=bibintoc]}

\usepackage[inline,shortlabels]{enumitem}

\usepackage{subcaption}

\usepackage[yyyymmdd]{datetime}

\usepackage{etoolbox}
\ifundef{\abstract}{}{\patchcmd{\abstract}%
    {\quotation}{\quotation\noindent\ignorespaces}{}{}}

\usepackage[super]{nth}

\usepackage{aliascnt}

\numberwithin{equation}{section}

\renewcommand{\eqref}[1]{\hyperref[#1]{\rm(\ref*{#1})}}

\def\makeautorefname#1#2{\AtBeginDocument{\expandafter\def\csname#1autorefname\endcsname{#2}}}

\newcommand{\mynewtheorem}[2]{
  \newaliascnt{#1}{equation}          
  \newtheorem{#1}[#1]{#2}
  \aliascntresetthe{#1}
  \makeautorefname{#1}{#2}
}

\mynewtheorem{axiom}{Axiom}
\newtheorem*{axiom*}{Axiom}
\mynewtheorem{theorem}{Theorem}
\newtheorem*{theorem*}{Theorem}
\mynewtheorem{prop}{Proposition}
\newtheorem*{prop*}{Proposition}

\mynewtheorem{cor}{Corollary}
\mynewtheorem{construction}{Construction}
\mynewtheorem{lemma}{Lemma}
\mynewtheorem{conjecture}{Conjecture}
\newtheorem*{conjecture*}{Conjecture}
\mynewtheorem{hyp}{Hypothesis}
\newtheorem{step}{Step}
\newtheorem{substep}{Step}
\numberwithin{substep}{step}
\makeautorefname{step}{Step}
\makeautorefname{substep}{Step}
\newtheorem{case}{Case}

\numberwithin{subcase}{case}
\makeautorefname{case}{Case}
\makeautorefname{subcase}{case}

\usepackage{etoolbox}
\AtBeginEnvironment{proof}{\setcounter{step}{0}}

\theoremstyle{remark}
\mynewtheorem{remark}{Remark}
\newtheorem*{remark*}{Remark}
\mynewtheorem{convention}{Convention}
\newtheorem*{convention*}{Convention}
\newtheorem*{conventions*}{Conventions}

\theoremstyle{remark}
\mynewtheorem{warning}{Warning}

\theoremstyle{definition}
\mynewtheorem{definition}{Definition}
\newtheorem*{definition*}{Definition}
\mynewtheorem{notation}{Notation}
\mynewtheorem{data}{Data}
\mynewtheorem{example}{Example}
\newtheorem*{example*}{Example}
\mynewtheorem{exercise}{Exercise}
\mynewtheorem{solution}{Solution}
\mynewtheorem{question}{Question}
\mynewtheorem{problem}{Problem}
\newtheorem*{question*}{Question}

\mynewtheorem{summary}{Summary}

\makeautorefname{table}{Table}        
\makeautorefname{chapter}{Chapter}
\makeautorefname{section}{Section}
\makeautorefname{subsection}{Section}
\makeautorefname{subsubsection}{Section}
\makeautorefname{footnote}{Footnote}
\AtBeginDocument{\def\itemautorefname~#1\null{(#1)\null}}
\AtBeginDocument{\def\equationautorefname~#1\null{(#1)\null}}

\usepackage{bm}
\usepackage{mathtools} 
\usepackage{stmaryrd} 

\DeclareFontFamily{U}{mathx}{\hyphenchar\font45}
\DeclareFontShape{U}{mathx}{m}{n}{
      <5> <6> <7> <8> <9> <10>
      <10.95> <12> <14.4> <17.28> <20.74> <24.88>
      mathx10
      }{}
\DeclareSymbolFont{mathx}{U}{mathx}{m}{n}
\DeclareFontSubstitution{U}{mathx}{m}{n}
\DeclareMathAccent{\widecheck}{0}{mathx}{"71}
\DeclareMathAccent{\wideparen}{0}{mathx}{"75}

\DeclareMathOperator{\End}{End}

\DeclareMathOperator{\HF}{\HF}

\DeclareMathOperator{\Hom}{Hom}

\DeclareMathOperator{\area}{area}

\DeclareMathOperator{\supp}{supp}

\DeclarePairedDelimiter{\Abs}{\|}{\|}

\DeclarePairedDelimiter{\abs}{\lvert}{\rvert}
\DeclarePairedDelimiter{\braket}{\langle}{\rangle}

\DeclarePairedDelimiter{\set}{\lbrace}{\rbrace}
\def\({\left(}
\def\){\right)}
\def\<{\left\langle}
\def\>{\right\rangle}

\newcommand{\Gtwo}{G_2}

\newcommand{\N}{{\mathbf{N}}}

\newcommand{\Q}{\mathbf{Q}}

\newcommand{\R}{\mathbf{R}}

\newcommand{\SU}{\mathrm{SU}}

\newcommand{\Spin}{\mathrm{Spin}}

\newcommand{\Z}{\mathbf{Z}}

\newcommand{\co}{\mskip0.5mu\colon\thinspace}

\newcommand{\defined}[2][\key]{\def\key{#2}\textbf{#2}\index{#1}}
\newcommand{\delbar}{\bar{\del}}

\newcommand{\del}{\partial}

\newcommand{\id}{\mathrm{id}}

\newcommand{\inner}[2]{\braket{#1, #2}}

\newcommand{\iso}{\cong}

\newcommand{\loc}{\mathrm{loc}}

\newcommand{\qandq}{\quad\text{and}\quad}

\newcommand{\qand}{\quad\text{and}}

\newcommand{\sing}{\mathrm{sing}}

\newcommand{\so}{\mathfrak{so}}

\newcommand{\vol}{\mathrm{vol}}

\newcommand{\wto}{\rightharpoonup}

\renewcommand{\H}{\mathbf{H}}
\renewcommand{\Im}{\operatorname{Im}}

\renewcommand{\P}{\mathbf{P}}

\renewcommand{\emptyset}{\varnothing}
\renewcommand{\epsilon}{\varepsilon}
\renewcommand{\setminus}{{\backslash}}

\renewcommand{\leq}{\leqslant}
\renewcommand{\geq}{\geqslant}

\makeatletter
\renewcommand*\env@matrix[1][*\c@MaxMatrixCols c]{%
  \hskip -\arraycolsep
  \let\@ifnextchar\new@ifnextchar
  \array{#1}}

\renewcommand\xleftrightarrow[2][]{%
  \ext@arrow 9999{\longleftrightarrowfill@}{#1}{#2}}
\newcommand\longleftrightarrowfill@{%
  \arrowfill@\leftarrow\relbar\rightarrow}
\makeatother

\newcommand{\rd}{{\rm d}}

\newcommand{\sfF}{{\sf F}}

\newcommand{\ur}{{\underline r}}

\newcommand{\uv}{{\underline v}}

\newcommand{\cH}{\mathcal{H}}

\newcommand{\cM}{\mathcal{M}}

\newcommand{\sE}{\mathscr{E}}

\newcommand{\sR}{\mathscr{R}}

\newcommand{\fd}{{\mathfrak d}}
\newcommand{\fe}{{\mathfrak e}}

\newcommand{\fri}{{\mathfrak i}}

\newcommand{\fp}{{\mathfrak p}}

\newcommand{\fs}{{\mathfrak s}}

\newcommand{\fz}{{\mathfrak z}}

\newcommand{\fF}{{\mathfrak F}}

\newcommand{\fH}{{\mathfrak H}}

\newcommand{\fX}{{\mathfrak X}}
\newcommand{\fY}{{\mathfrak Y}}

\newcommand{\slS}{\slashed S}

\author{
  Thomas Walpuski
}

\title{A compactness theorem for Fueter sections}
\date{2017-03-01}
\begin{document}

\maketitle

\begin{abstract}
  We prove that a sequence of Fueter sections of a bundle of compact hyperkähler manifolds $\fX$ over a $3$--manifold $M$ with bounded energy converges (after passing to a subsequence) outside a $1$--dimensional closed rectifiable subset $S \subset M$.
  The non-compactness along $S$ has two sources:
  (1) Bubbling-off of holomorphic spheres in the fibres of $\fX$ transverse to a subset $\Gamma \subset S$, whose tangent directions satisfy strong rigidity properties.
  (2) The formation of non-removable singularities in a set of $\cH^1$--measure zero.
  Our analysis is based on the ideas and techniques that Lin developed for harmonic maps \cite{Lin1999}.
  These methods also apply to Fueter sections on $4$--dimensional manifolds; we discuss the corresponding compactness theorem in an appendix.
  We hope that the work in this paper will provide a first step towards extending the hyperkähler Floer theory developed by \citet{Hohloch2009,Salamon2013} to general target spaces.
  Moreover, we expect that this work will find applications in gauge theory in higher dimensions.
\end{abstract}

\section{Introduction}
\label{Sec_Introduction}

Let $M$ be an orientable Riemannian $3$--manifold, let $\fX \xrightarrow{\pi} M$ be a bundle of hyperkähler manifolds together with a fixed isometric identification $I \co STM \to \fH(\fX)$ of the unit tangent bundle in $M$ and the bundle of hyperkähler spheres\footnote{Given a hyperkähler manifold $(X,g,I_1,I_2,I_3)$, for each $\xi = (\xi_1,\xi_2,\xi_3) \in S^2 \subset \R^3$, $I_\xi := \sum_{i=1}^3 \xi_i I_i$ is a complex structure.
The set $\fH(X) := \set{ I_\xi : \xi \in S^2 }$ is called the \defined{hyperkähler sphere} of $X$.} of the fibres of $\fX$, and fix a connection on $\fX$.

\begin{definition}
  A section $u \in \Gamma(\fX)$ is called a \defined{Fueter section} if
  \begin{equation}
    \label{Eq_Fueter}
    \fF u := \sum_{i=1}^3 I(v_i) \nabla_{v_i} u = 0 \in \Gamma(u^*V\fX)
  \end{equation}
  for some local orthonormal frame $(v_1,v_2,v_3)$.\footnote{Of course, $\fF$ does not depend on the choice of $(v_1,v_2,v_3)$.}
  Here $\nabla u \in \Omega^1(M,u^*V\fX)$ is the covariant derivative of $u$, a $1$--form taking values in the pull-back of the vertical tangent bundle $V\fX := \ker\(\rd\pi \co T\fX \to TM\)$.
  The operator $\fF$ is called the \defined{Fueter operator}.
\end{definition}

The Fueter operator is a non-linear generalisation of the Dirac operator, see \citet{Taubes1999b} and \citet[Section 3]{Haydys2013}.

\begin{remark}
  A construction similar to \eqref{Eq_Fueter} also exists in dimension four.
  Since it is more involved, we relegate its discussion to \autoref{Sec_DimensionFour}.
\end{remark}

\begin{example}
  Choose a spin structure $\fs$ on $M$.
  If $\fX = \slS$, $I$ is the Clifford multiplication and $\nabla$ denotes the induced spin connection, then the Fueter operator is simply the Dirac operator associated with $\fs$.
\end{example}

\begin{example}
  \label{Ex_FueterMap}
  Let $(X,g,I_1,I_2,I_3)$ be a hyperkähler manifold and $(v_1,v_2,v_3)$ a orthonormal frame of $M$.
  A map $u\co M \to X$ satisfying
  \begin{equation}
    \label{Eq_FueterMap}
    \fF u = \sum_{i=1}^3 I_i \rd u(v_i) = 0
  \end{equation}
  is called a Fueter map.
  In a local trivialisation the Fueter equation for sections of $\fX$, takes the form \eqref{Eq_FueterMap} up to allowing for the $I_i$ to depend on $x \in M$ and admitting a lower order perturbation (coming from the connection $1$--form).
\end{example}

One of the main motivations for studying Fueter sections is the work of \citet{Hohloch2009}, who introduced a functional whose critical points are precisely the solution of \eqref{Eq_FueterMap} and developed the corresponding Floer theory in the case when the target $X$ is compact and flat, and the frame on $M$ is divergence free and regular,\footnote{%
  Every $3$--manifold admits a divergence free frame by Gromov's h--principle \cite[Theorem A.1]{Salamon2013}.
  A frame is regular if there are no non-constant Fueter maps $M \to \H$ with respect to this frame; this is a generic condition,.
}
see also \citet{Salamon2013}.
The requirement that $X$ be flat is very severe and one would like to remove it.
It has been conjectured that the putative hyperkähler Floer theory should be very rich and interesting, especially in the case when $X$ is a $K3$ surface.

A further source of motivation is gauge theory on $\Gtwo$-- and $\Spin(7)$--manifolds.
Here, Fueter sections of bundles of moduli spaces of ASD instantons naturally appear in relation with codimension four bubbling phenomena for $\Gtwo$-- and $\Spin(7)$--instantons; see Donaldson--Segal~\cite{Donaldson2009} and the author \cites{Walpuski2013a,Walpuski2014} for further details.

\begin{remark}
  Sonja Hohloch brought to the author's attention a cryptic remark in \citet[Section 1.5 Question 3]{Kontsevich2008}, which indicates that their invariants of 3D Calabi--Yau categories with stability structure can be interpreted as ``quaternionic Gromov--Witten invariants'' of certain hyperkähler manifold $\cM$, which means as a count of Fueter maps from some $4$--manifold to $\cM$.
\end{remark}

A major issue when dealing with Fueter sections is the potential failure of compactness.
This is demonstrated by the following example due to Hohloch, Noetzel, and Salamon.

\begin{example}
  \label{Ex_HNS}
  Consider a $K3$ surface $X$ with a hyperkähler structure such that $(X,I_1)$ admits a non-trivial holomorphic sphere $\fz \co S^2 \to X$ and take $M = \SU(2)$, the unit-sphere in the quaternions $\H$, with a left-invariant frame $(v_1,v_2,v_2)$ which at $\id \in \SU(2)$ it is given by $(i,j,k)$.
  Let $\bar\cdot \co S^2 \to S^2$ denote complex conjugation on $S^2 = \P^1$.
  Let $\pi \co S^3 \to S^2$ denote the Hopf fibration whose fibres are the orbits of $v_1$.
  It is easy to check that $u = \fz \circ \bar\cdot \circ \pi \co S^3 \to X$ satisfies
  \begin{equation*}
    \del_{v_1} u = 0 \qandq
    \del_{v_2} u - I\del_{v_3} u = 0,
  \end{equation*}
  and thus $u$ is a Fueter map.
  For $\lambda > 0$ define a conformal map $s_\lambda \co S^2 \to S^2$ by $s_\lambda(x) = \lambda x$ for $x \in \R^2 \subset S^2$ and $s_\lambda(\infty) = \infty$.
  Now, the family of Fueter maps $u_\lambda := \fz \circ s_\lambda \circ \pi$ blows up along the Hopf circle $\pi^{-1}(\infty)$ as $\lambda \downarrow 0$ and converges to the constant map on the complement of the Hopf circle.
  Also, note that $\sE(u_\lambda) = \int_{S^3} |\nabla u_\lambda|^2$ is independent of $\lambda$.
\end{example}

The following is the main result of this article.

\begin{theorem}
  \label{Thm_A}
  Suppose $\fX$ is compact.
  Let $(u_i)$ be a sequence of solutions of the (perturbed) Fueter equation
  \begin{equation}
    \label{Eq_FueterPerturbed}
    \fF u_i = \fp \circ u_i
  \end{equation}
  with $\fp \in \Gamma(\fX,V\fX)$\footnote{%
    This sort of deformation of \eqref{Eq_Fueter} is important for applications; e.g., Hohloch, Noetzel, and Salamon perturb \eqref{Eq_Fueter} using a Hamiltonian function to achieve transversality.} and
  \begin{equation}
    \label{Eq_EnergyBound}
    \sE(u_i) := \int_{M} |\nabla u_i|^2 \leq c_\sE
  \end{equation}
  for some constant $c_\sE > 0$.
  Then (after passing to a subsequence) the following holds:
  \begin{itemize}
  \item
    There exists a closed subset $S$ with $\cH^1(S) < \infty$ and a Fueter section $u \in \Gamma(M \setminus S, \fX)$ such that $u_i|_{M\setminus S}$ converges to $u$ in $C^\infty_\loc$.
  \item 
    There exist a constant $\epsilon_0 > 0$ and an upper semi-continuous function $\Theta \co S \to [\epsilon_0,\infty)$ such that the sequence of measures $\mu_i := |\nabla u_i|^2 \,\cH^3$ converges weakly to $\mu = |\nabla u|^2 \,\cH^3 + \Theta \,\cH^1\lfloor S$.
  \item
    $S$ decomposes as
    \begin{equation*}
      S = \Gamma \cup \sing(u)
    \end{equation*}
    with
    \begin{align*}
      \Gamma
      &:=
        \supp(\Theta \,\cH^1\lfloor S) \qand \\
      \sing(u)
      &:=
        \set*{
          x \in M
          :
          \limsup_{r\downarrow 0} \frac1r \int_{B_r(x)} |\nabla u|^2 > 0
        }.
    \end{align*}
    $\Gamma$ is $\cH^1$--rectifiable, and  $\cH^1(\sing(u)) = 0$.
  \item
    For each smooth point\footnote{We call a point $x \in \Gamma$ smooth if the tangent space $T_x\Gamma$ exists and $x \notin \sing(u)$.  Since $\Gamma$ is rectifiable, $T_x\Gamma$ exists almost everywhere.} $x \in \Gamma$, there exists a non-trivial holomorphic sphere $\fz_x \co S^2 \to (\fX_x := \pi^{-1}(x),-I(v))$ with $v$ a unit tangent vector in $T_x\Gamma$.
    Moreover,
    \begin{equation*}
      \Theta(x) \geq \sE(\fz_x) := \int_{S^2} |\rd \fz_x|^2.
    \end{equation*}
  \item
    If $\fX$ is a bundle of simple hyperkähler manifolds with $b_2 \geq 6$, then there is a subbundle $\fd \subset \P TM$, depending only on $\sup \Theta$, whose fibres are finite sets such that $T_x\Gamma \in \fd$ for all smooth points $x \in \Gamma$.
  \end{itemize}
\end{theorem}

\begin{remark}
  The analysis of \eqref{Eq_Fueter} is similar to Lin's work on the compactness problem for harmonic maps~\cite{Lin1999}.
  We follow his strategy quite closely;
  however, there are a number of simplifications in our case, many of the arguments have to be approached from a different angle and our result is stronger. 
\end{remark}

\begin{remark}
  In the situation of \autoref{Ex_FueterMap} if $X$ is flat and $(v_1,v_2,v_3)$ is regular, then the uniform energy bound \eqref{Eq_EnergyBound} is automatically satisfied;
  see \citet[Lemma 3.2 and Remark 3.5]{Salamon2013}.
\end{remark}

\begin{remark}
  \label{Rmk_HarmonicSections}
  If $I$ is parallel (which is very rarely the case, but holds, e.g., in the situation of \autoref{Ex_FueterMap} if $M = T^3$ equipped with a flat metric and the $v_i$ are parallel), then there are topological energy bounds;
  see \autoref{Rmk_LambdaClosedEnergyBound}.
  In this case Fueter sections are stationary harmonic sections and one can derive most of \autoref{Thm_A} from \cite{Lin1999};
  cf.~\citet[Section 4]{Li1998} and \citet{Chen2000}, who study triholomorphic/quaternionic maps between hyperkähler manifolds.
  More recently, very important progress in the study of triholomorphic maps was made by \citet{Bellettini2015}.
\end{remark}

\begin{remark}
  In the situation of \autoref{Ex_FueterMap} if $X$ is flat, then $S = \emptyset$;
  see \citet[Section 3]{Hohloch2009} and \autoref{Rmk_FlatEpsilonRegularity}.
  This does not immediately follow from \autoref{Thm_A};
  however, since $\pi_2(T^n) = 0$, flat hyperkähler manifolds admit no non-trivial holomorphic spheres and we can rule out bubbling a priori, i.e., $\Gamma = \emptyset$.
  See also \autoref{Rmk_FlatEpsilonRegularity}.
\end{remark}

\begin{remark}
  By Bogomolov's decomposition theorem (after passing to a finite cover) any hyperkähler manifold is a product a flat torus and simple hyperkähler manifolds.
  Hohloch, Noetzel, and Salamon's compactness result says that nothing interesting happens in the torus-factors.
  Thus the assumption of $\fX$ being a bundle of simple hyperkähler manifolds is not restrictive.
  The requirement $b_2 \geq 6$ is an artefact of a result of Amerik and Verbitsky that we use in \autoref{Sec_TangentDirections}.
\end{remark}

As stated, \autoref{Thm_A} is very likely far from optimal.
Here are some conjectural improvements:
\begin{itemize}
\item 
  We believe that the limiting section $u \in \Gamma(M\setminus S,\fX)$ extends to $M\setminus \sing(u)$ and, moreover, that $\sing(u)$ is finite (possibly countable and closed).
\item
  We believe that $\Gamma$ enjoys much better regularity than just being $\cH^1$--rectifiable.
  It seems reasonable to expect that $\Gamma$ is a graph (possibly with countably many vertices) embedded in $M$ and $\Theta$ is constant along the edges of $\Gamma$;
  moreover, we expect that the vertices $(\Gamma,\Theta)$ are balanced.
\end{itemize}

\begin{remark}
  In the situation of \autoref{Rmk_HarmonicSections}, Bethuel's removable singularities theorem for stationary harmonic maps \cite[Theorem I.4]{Bethuel1993} shows that $u$ extends to $M \setminus \sing(u)$ and a result of \citet{Allard1976} affirms the conjecture in the third bullet.
\end{remark}

The holomorphic sphere $\fz_x$ can be replaced by a bubble-tree, cf.~\citet{Parker1993}, such that the energy of the entire bubble tree equals $\Theta(x)$.
In an earlier version of this article it was conjectured that there can be no energy stuck on the necks;
in, particular $\Theta(x)$ is the sum of energies of holomorphic spheres in $(\fX_x,-I(v))$.
Shortly after the first version of this article was posted on the arXiv,
\citet{Bellettini2015} proved the analogue of this conjecture for triholomorphic maps, and after a brief discussion with the author, in an updated version also the author's earlier conjecture.
We refer the reader to \cite[Section 7]{Bellettini2015} for details.

It is an interesting and important question to ask: what happens for a generic choice of $I\co STM \to \fH(\fX)$ and perturbation $\fp$?
One would hope (perhaps too optimistically) that generically the situation is much better and possibly good enough to count solutions of \eqref{Eq_FueterPerturbed} and thus define the Euler characteristic of the conjectural hyperkähler Floer theory.

\paragraph{Assumptions and conventions}
Throughout the rest of the article we assume the hypotheses of \autoref{Thm_A}.
We use $c$ to denote a generic constant.
We write $x \lesssim y$ for $x \leq c y$ and $\{\cdot,\ldots,\cdot\}$ denotes a generic (multi-)linear expression which is bounded by $c$.
We fix a constant $0 < r_0 \ll 1$; in particular, $r_0$ is much smaller than the injectivity radius of $M$ and we take all radii to be at most $r_0$.


\section{Mononicity formula}
\label{Sec_Monotonicity}

The foundation of the analysis of \eqref{Eq_Fueter} is the monotonicity formula which asserts that the renormalised energy
\begin{equation*}
  \frac1r \int_{B_r(x)} \abs{\nabla u}^2.
\end{equation*}
is almost monotone in $r > 0$:

\begin{prop}
  \label{Prop_Monotonicity}
  If $u \in \Gamma(M, \fX)$ satisfies \eqref{Eq_FueterPerturbed}, then for all $x \in M$ and $0 < s < r \leq r_0$
  \begin{equation*}
    \label{Eq_Monotonicity}
    \frac{e^{cr}}{r} \int_{B_r(x)} \abs{\nabla u}^2
    - \frac{e^{cs}}{s} \int_{B_s(x)} \abs{\nabla u}^2
    \geq
      \int_{B_r(x)\setminus B_s(r)} \frac1\rho \abs{\nabla_r u}^2 - c(r^2 - s^2).
  \end{equation*}
  Here $\rho := d(x,\cdot)$.
\end{prop}

It is instructive to first prove the following which contains the essence of \autoref{Prop_Monotonicity}.

\begin{prop}
  \label{Prop_FlatMonotonicity}
  If $u\co \R^3 \to X$ is a Fueter map with $v_i = \frac{\del}{\del x^i}$, then for all $x \in M$ and $0 < s < r$
  \begin{equation}
    \label{Eq_FlatMonotonicity}
    \frac{1}{r} \int_{B_r(x)} \abs{\rd u}^2
    - \frac{1}{s} \int_{B_s(x)} \abs{\rd u}^2
    =
      2\int_{B_r(x)\setminus B_s(r)} \frac1\rho \abs{\del_r u}^2.
  \end{equation}
\end{prop}

\begin{proof}
  The derivative of
  \begin{equation*}
    f(\rho) := \frac1\rho \int_{B_\rho(x)} \abs{\rd u}^2
  \end{equation*}
  is
  \begin{equation*}
    f'(\rho) = -\frac1{\rho^2} \int_{B_\rho(x)} \abs{\rd u}^2 + \frac1\rho \int_{\del B_\rho(x)} \abs{\rd u}^2.
  \end{equation*}
  By a direct computation
  \begin{equation}
    \label{Eq_EnergyFlat}
    \abs{\rd u}^2 \,\vol =  \abs{\fF u}^2 \,\vol - 2\sum_{i=1}^3 \rd x^i \wedge u^*\omega_i,
  \end{equation}
  see \cite[Lemma 2.2]{Hohloch2009}.
  Here $\omega_i = g(I_i\,\cdot,\cdot)$ denotes the Kähler form on $X$ associated with $I_i$.
  Hence,
  \begin{equation}
    \label{Eq_Stokes}
    \begin{split}
      -\int_{B_\rho(x)} \abs{\rd u}^2
      &=
        2\int_{B_\rho(x)} \sum_{i=1}^3 \rd x^i \wedge u^*\omega_i
      =
        2\int_{B_\rho(x)} \sum_{i=1}^3 \rd (x^i u^*\omega_i) \\
      &=
        2\rho \int_{\del B_\rho(x)} u^*\omega_{\del_r}
    \end{split}
  \end{equation}
  with $\del_r = \sum_{i=1}^3 \frac{x^i}{|x|} \frac{\del}{\del x^i}$ denoting the radial vector field.
  On $\del B_\rho(x)$, we can take the local orthonormal frame $(v_1,v_2,v_3)$ to be of the form $(\del_r,\del_1,\del_2)$ with $(\del_1, \del_2)$ a local positive orthonormal frame for $\del B_\rho(x)$.
  Now, twice the integrand in the last term is
  \begin{equation}
    \label{Eq_BoundaryContribution}
    \begin{split}
      2\inner{I(\del_r)\del_1u}{\del_2u}
      &= 2\inner{I_1\del_1u}{I_2\del_2u} \\
      &= \abs{I_1\del_1u + I_2\del_2u}^2 - \abs{I_1\del_1u}^2 - \abs{I_2\del_2u}^2 \\
      &= 2\abs{\del_r u}^2 - \abs{\rd u}^2.
    \end{split}
  \end{equation}
  Putting everything together yields
  \begin{equation*}
    f'(\rho) = 2\rho^{-1}\int_{\del B_r} \abs{\del_r u}^2.
  \end{equation*}
  Upon integration this yields \eqref{Eq_FlatMonotonicity}.
\end{proof}

\begin{proof}[Proof of \autoref{Prop_Monotonicity}]
  The map $I$ yields a section of $\pi^*TM \otimes \Lambda^2 V\fX$ which, using the connection on $\fX$, can be viewed as a $3$--form $\Lambda \in \Omega^3(\fX)$.
  For sections of $\fX$ the identity \eqref{Eq_EnergyFlat} is replaced by
 \begin{equation}
    \label{Eq_Energy}
    \abs{\nabla u}^2 \,\vol = \abs{\fF u}^2 \,\vol - 2u^*\Lambda.
  \end{equation}
  If we define $f(\rho)$ as before, then using \eqref{Eq_Energy} its derivative can be written as
  \begin{equation*}
    f'(\rho)
    = -\rho^{-2}\int_{B_\rho(x)} \abs{\fp\circ u}^2
    + 2\rho^{-2} \int_{B_\rho(x)} u^*\Lambda
    + \rho^{-1} \int_{\del B_\rho(x)} \abs{\nabla u}^2.
  \end{equation*}

  Let $\del_r$ denote the radial vector field emanating from $x$ and set $\Omega := i(\uv)\Lambda$ with $\uv := \pi^*(r\del_r)$.
  We can write $\Lambda$ as
  \begin{equation*}
    \Lambda = \rd\Omega + \fe
  \end{equation*}
  where $\fe$ is the sum of a form of type $(1,2)$ and a form of type $(2,1)$ satisfying
  \begin{equation}
    \label{Eq_Delta}
    \abs{\fe} = O(\delta\ur) \quad\text{with}\quad
    \delta := \abs{\nabla I} + \abs{F_\fX} + \abs{R}.
  \end{equation}
  Here we use the bi-degree decomposition of $\Omega^*(\fX)$ arising from $T\fX = \pi^*TM \oplus V\fX$, $\ur := d(x,\pi(\cdot))$, $F_\fX$ is the curvature of the connection on $\fX$ and $R$ is the Riemannian curvature of $M$.
  Hence,
  \begin{equation}
    \label{Eq_StokesSection}
    \begin{split}
      2\int_{B_\rho(x)} u^*\Lambda
      &=
        2\int_{\del B_\rho(x)} u^*\Omega + O(\rho^2) f(\rho) + O\(\rho^4\) \\
      &=
        2\rho\int_{\del B_\rho(x)} i(\del_r)u^*\Lambda +  O(\rho^2) f(\rho) + O\(\rho^4\).
    \end{split}
  \end{equation}
  Arguing as before,
  \begin{equation*}
    2\int_{\del B_\rho(x)} i(\del_r)u^*\Lambda
    =
      \int_{\del B_\rho(x)} \abs{I_{\del_r}\nabla_r u - \fp\circ u}^2 + \abs{\nabla_r u}^2 - \abs{\nabla u}^2.
  \end{equation*}
  Putting everything together one obtains
  \begin{equation*}
    f'(\rho) \geq \frac{1}{\rho} \int_{B_\rho(x)} \abs{\nabla_r u}^2 - c f(\rho) - c \rho.
  \end{equation*}
  Since we can assume that $e^{cr_0} \leq 2$ and using $e^{c\rho} \geq 1$, we have
  \begin{equation*}
    \del_\rho(e^{c\rho}f(\rho))
    \geq
    \frac{1}{\rho}\int_{B_\rho(x)} \abs{\nabla_r u}^2
    - 2c\rho
  \end{equation*}
  This integrates to prove the assertion.
\end{proof}

\begin{remark}
  \label{Rmk_LambdaClosedEnergyBound}
  If $\Lambda$ is closed (which is rarely the case), then
  \begin{equation*}
    \sE(u)
    =
      \int_M \abs{\nabla u}^2
    =
      \int_M \abs{\fF u}^2 - 2\<[M],[u^*\Lambda]\>.
  \end{equation*}
  Since the first term on the right-hand side only depends on the homotopy class of $u$, this yields a priori energy bounds for Fueter sections.
\end{remark}

\begin{cor}
  \label{Cor_Monotonicity}
  In the situation of \autoref{Prop_Monotonicity},
  \begin{equation*}
    \frac{1}{s} \int_{B_s(x)} \abs{\nabla u}^2
    \lesssim
      \frac{1}{r} \int_{B_r(x)} \abs{\nabla u}^2 + r^2
  \end{equation*}
  and if $B_s(y) \subset B_{r/2}(x)$, then
  \begin{equation*}
    \frac{1}{s} \int_{B_s(y)} \abs{\nabla u}^2
    \lesssim
      \frac{1}{r} \int_{B_r(x)} \abs{\nabla u}^2 + r^2.
  \end{equation*}
\end{cor}


\section{\texorpdfstring{$\epsilon$}{epsilon}--regularity}
\label{Sec_EpsilonRegularity}

The following is the key result for proving \autoref{Thm_A}.
It allows to obtain local $L^\infty$--bounds on $\nabla u$ provided the renormalised energy is not too large.

\begin{prop}
  \label{Prop_EpsilonRegularity}
  There is a constant $\epsilon_0 > 0$ such that if $u \in \Gamma(M,\fX)$ satisfies \eqref{Eq_FueterPerturbed} and
  \begin{equation*}
    \epsilon := \frac1r \int_{B_r(x)} \abs{\nabla u}^2 \leq \epsilon_0,
  \end{equation*}
  then
  \begin{equation}
    \label{Eq_LinftyBound}
    \sup_{y \in B_{r/4}(x)} \abs{\nabla u}^2(y) \lesssim r^{-2}\epsilon + 1.
  \end{equation}
\end{prop}

\begin{remark}
  \label{Rmk_HigherEpsilonRegularity}
  Given \eqref{Eq_LinftyBound}, higher derivative bounds over slightly smaller balls can be obtained using interior elliptic estimates.  
\end{remark}

\autoref{Prop_EpsilonRegularity} follows from the following differential inequality and \autoref{Cor_Monotonicity} using the Heinz trick;
see \autoref{A_HeinzTrick}.

\begin{prop}
  \label{Prop_DifferentialInequality}
  If $u \in \Gamma(M,\fX)$ satisfies \eqref{Eq_FueterPerturbed}, then
  \begin{equation*}
    \Delta\abs{\nabla u}^2 \lesssim \abs{\nabla u}^4 + 1.
  \end{equation*}
\end{prop}

\begin{proof}
  This is proved in \cite[Lemma 3.3 and Remark 3.4]{Hohloch2009}.
  We recall the proof which is a simple direct computation.
  Denote by $\bar \nabla$ the induced connection on $u^*V\fX$ and define $\sfF \co \Omega^0(M,u^*V\fX) \to \Omega^0(M,u^*V\fX)$ by
  \begin{equation*}
    \sfF \hat u := \sum_{i=1}^3 I(v_i) \bar\nabla_{v_i} \hat u
  \end{equation*}
  for some local orthonormal frame $(v_1,v_2,v_3)$.
  A simple computation yields
  \begin{equation*}
    \sfF\fF u = \bar\nabla^*\nabla u + \{\nabla u\}
  \end{equation*}
  where $\{\cdot\}$ makes the dependence on $I$ etc.~implicit.
  Further
  \begin{align*}
    \bar\nabla\sfF\fF u = \bar\nabla\bar\nabla^*\nabla u + \{\nabla u\} + \{\bar\nabla\nabla u\}.
  \end{align*}
  Using
  \begin{align*}
    \bar\nabla_{v_k} \bar\nabla_{v_i} \nabla_{v_i} u
    &= \bar\nabla_{v_i} \bar\nabla_{v_k} \nabla_{v_i} u + \{\nabla u,\nabla u,\nabla u\} \\
    &= \bar\nabla_{v_i} \bar\nabla_{v_i} \nabla_{v_k} u + \{\nabla u,\nabla u,\nabla u\} + \{\bar\nabla\nabla u\}
  \end{align*}
  and $\fF u = \fp \circ u$ we derive
  \begin{align*}
    \bar\nabla^*\bar\nabla\nabla u
    &= \bar\nabla \sfF\fF u + \{\nabla u,\nabla u, \nabla u\} + \{\bar\nabla\nabla u\} \\
    &= \{\nabla u,\nabla u, \nabla u\} + \{\bar\nabla\nabla u\} + O(1).
  \end{align*}
  From this it follows that
  \begin{align*}
    \Delta \abs{\nabla u}^2
    &= 2\<\bar\nabla^*\bar\nabla \nabla u, \nabla u\> - 2|\bar\nabla\nabla u|^2 \\
    &\leq c (\abs{\nabla u}^4 + \abs{\nabla u} + |\bar\nabla\nabla u||\nabla u|^2) - 2 \abs{\bar\nabla\nabla u}^2 \\
    &\lesssim \abs{\nabla u}^4 + 1.
    \qedhere
  \end{align*}
\end{proof}

\begin{remark}
  \label{Rmk_FlatEpsilonRegularity}
  If $\fX = M \times X$ and $X$ is flat, then one can prove that
  \begin{equation*}
    \Delta |\nabla u|^2 \lesssim |\nabla u|^3 + 1
  \end{equation*}
  and the Heinz trick for subcritical exponents shows that $\|\nabla u\|_{L^\infty(M)}$ is bounded in terms of the energy $\sE(u)$;
  see \autoref{Rmk_SubcriticalHeinzTrick} and \cite[Appendix B]{Hohloch2009}.
\end{remark}


\section{Convergence away from the blow-up locus}
\label{Sec_ConvergenceAwayFromBlowupLocus}

\begin{prop}
  \label{Prop_ConvergenceAwayFromBlowupLocus}
  There exists a subsequence $(u_i)_{i \in I} \subset (u_i)_{i \in \N_0}$ and a subset $S \subset M$, called the \defined{blow-up locus}, with the following properties:
  \begin{itemize}
  \item
    $S$ is closed and $\cH^1(S) < \infty$.
  \item
    The sequence $\(u_{i}|_{M\setminus S}\)_{i \in I}$ converges to a section $u \in \Gamma(M\setminus S, \fX)$ in $C^\infty_\loc$.
  \item
    If there is a subset $S' \subset M$ such that a subsequence $\(u_i|_{M\setminus S'}\)_{i \in I' \subset I}$ converges in $C^\infty_\loc$, then $S' \supset S$.
  \end{itemize}
\end{prop}

\begin{proof}
  We proceed in four steps.
  \setcounter{step}{0}
  \begin{step}
    Construction of $S$.
  \end{step}

  With $\epsilon_0$ as in \autoref{Prop_EpsilonRegularity}, for $r \in (0,r_0]$ and $i \in \N_0$, define
  \begin{equation*}
    S_{i,r}
    := \set*{
      x \in M : \frac{e^{cr}}{r} \int_{B_r(x)} \abs{\nabla u_i}^2 + cr^2 \geq \frac{\epsilon_0}{2}
    }.
  \end{equation*}
  Note that, by \autoref{Prop_Monotonicity}, $S_{i,s} \subset S_{i,r}$ whenever $s \leq r$.

  Since the $S_{i,r}$ are compact, for each $r$, we can pick $J_r \subset \N_0$ such that the subsequence $(S_{i,r})_{i \in J_r}$ converges to a closed subset $S_r$ in the Hausdorff metric.
  By a diagonal sequence argument, we can find $J \subset \N_0$ such that $\(S_{i,2^{-k}r_0}\)_{i \in J}$ converges to a closed subset $S_{2^{-k}r_0}$ for each $k \in \N_0$.
  Set
  \begin{equation*}
    S := \bigcap_{k \in \N_0} S_{2^{-k}r_0}.
  \end{equation*}
  By construction $S$ is closed.

  \begin{step}
    $\cH^1(S) < \infty$.
  \end{step}

  Given $0 < \delta \leq r_0$, cover $S$ by a collection of balls $\set{ B_{4r_j}(x_j) : j = 1, \ldots, m }$ with $x_j \in S$, $r_j \leq \delta$ and $B_{2r_j}(x_j)$ pairwise disjoint.
  Pick $k \gg 1$ such that $2^{-k}r_0 < \min\set{r_j}$.
  For $i \gg 1$, we can find $x_j' \in S_{i,2^{-k}r_0}$ with $d(x_j',x_j) < \delta$.
  Then the balls $B_{5r_j}(x_j')$ still cover $S$ while the smaller balls $B_{r_j}(x_j')$ are pairwise disjoint.
  By definition of $S_{i,r}$,
  \begin{align*}
    \sum_{j=1}^m r_j
    &\leq
      \frac{2e^{cr_0}}{\epsilon_0} \sum_{j=1}^m \int_{B_{r_j}(x_j')} \abs{\nabla u_i}^2 + cr_j^2 
    \leq
      \frac{2e^{cr_0}}{\epsilon_0} \int_M \abs{\nabla u_i}^2
      +  cr_0 \sum_{j=1}^m r_j.
  \end{align*}
  Since we can assume that $cr_0 \leq 1/2$ and $e^{cr_0} \leq 2$, it follows that
  \begin{equation*}
    \sum_{j=1}^m r_j
    \leq
      \frac{8c_\sE}{\epsilon_0}.
  \end{equation*}
  Since this bound is uniform in $\delta \in (0,r_0]$, the assertion follows.

  \begin{step}
    Selection of $(u_i)_{i\in I}$ and construction of $u \in \Gamma(M\setminus S, \fX)$.
  \end{step}

  If $x \in M\setminus S$, then there exists $r \in (0,r_0]$ such that for all $i \in J$ sufficiently large
  \begin{equation*}
    \frac{1}{r} \int_{B_r(x)} \abs{\nabla u_i}^2
    \leq \epsilon_0.
  \end{equation*}
  By \autoref{Prop_EpsilonRegularity}, for all $i \in J$, $\abs{\nabla u_i}$ is uniformly bounded on $B_{r/4}(x)$.
  It follows using standard elliptic techniques and Arzelà--Ascoli that we can chose $J \subset I$ such that the subsequence of $(u_i)_{i \in I}$ converges in $C^\infty_\loc$ on $M\setminus S$.

  \begin{step}
    $M\setminus S$ is the maximal open subset on which a subsequence $(u_i)_{i \in I' \subset I}$ can converge in $C^\infty_\loc$.
  \end{step}

  Suppose $(u_i)_{i \in I' \subset I}$ converges in $C^1$ in a neighbourhood of $x \in M$.
  Then $\abs{\nabla u_i}$ is uniformly bounded in this neighbourhood.
  Hence, there is a slightly smaller neighbourhood of $x \in M$ which is contained in $M \setminus S_{i,r}$ for each sufficiently small $r > 0$ and each $i \in I'$.
  Since $\lim_{i \in I'} S_{i,r} = S_r \subset S$, it follows that $x \in M \setminus S$.
\end{proof}


\section{Decomposition of the blow-up locus}
\label{Sec_DecompositonOfBlowupLocus}

We assume that we have already passed to a subsequence so that the convergence statement in \autoref{Prop_ConvergenceAwayFromBlowupLocus} holds.
Consider the sequence of measures $(\mu_i)$ defined by
\begin{equation*}
  \mu_i := \abs{\nabla u_i}^2 \,\cH^3.
\end{equation*}
Here $\cH^3$ is the $3$--dimensional Hausdorff measure on $M$, which is simply the standard measure on $M$.
By \eqref{Eq_EnergyBound} the sequence of Radon measures $(\mu_i)$ is of bounded mass;
hence, it converges weakly to a Radon measure $\mu$.
By Fatou's lemma we can write
\begin{equation*}
  \mu = \abs{\nabla u}^2\,\cH^3 + \nu
\end{equation*}
for some non-negative Radon measure $\nu$.

\begin{definition}
  We call $\nu$ the \defined{defect measure} and
  \begin{equation*}
    \Gamma := \supp \nu
  \end{equation*}
  the \defined{bubbling locus}.\footnote{%
    The justification for this terminology will be provided in \autoref{Sec_BubblingAnalysis}.
  }
  We call
  \begin{equation*}
    \sing(u) 
    := \set*{
      x \in M
      :
      \Theta^*_u(x) := \limsup_{r\downarrow 0} \frac1r \int_{B_r(x)} \abs{\nabla u}^2 > 0
    }
  \end{equation*}
  the \defined{singular set} of $u$.
\end{definition}

If we denote by $\Theta_\mu^*(x)$ the upper density of $\mu$ at the point $x \in M$, then it follows from \autoref{Prop_EpsilonRegularity} that $S = \set{ x \in M : \Theta^*_\mu(x) > 0 } \subset \Gamma \cup \sing(u)$.
The reverse inclusion also holds;
hence, we have the following.

\begin{prop}
  \label{Prop_Decomposition}
  The blow-up locus $S$ decomposes as
  \begin{equation*}
    S = \Gamma \cup \sing(u).
  \end{equation*}
\end{prop}

This means that there are two sources of non-compactness:
one involving a loss of energy and another one without any loss of energy.


\section{Regularity of the bubbling locus}
\label{Sec_RegularityOfBlowUpLocus}

As a first step towards understanding the non-compactness phenomenon involving energy loss, we show that the set $\Gamma$ at which this phenomenon occurs is relatively tame.

\begin{prop}
  \label{Prop_NuRectifiable}
  $\Gamma$ is $\cH^1$--rectifiable and $\nu$ can be written as
  \begin{equation*}
    \nu = \Theta \,\cH^1\lfloor \Gamma
  \end{equation*}
  with $\Theta \co M \to [0,\infty)$ upper semi-continuous. 
  Moreover, $\cH^1(\sing(u)) = 0$.
\end{prop}

The interested reader can find a detailed discussion of the concept of rectifiablity in DeLellis' lecture notes \cite{DeLellis2008}.
For our purposes it shall suffice to recall the definition.

\begin{definition}
  A subset $\Gamma \subset M$ is called \defined{$\cH^k$--rectifiable} if there exists a countable collection $\set{\Gamma_i}$ of $k$--dimensional Lipschitz submanifolds such that
  \begin{equation*}
    \cH^k\(\Gamma \setminus \bigcup\nolimits_i \Gamma_i\) = 0.
  \end{equation*}
  A measure $\mu$ on $M$ is called \defined{$\cH^k$--rectifiable} if there exist a non-negative Borel measurable function $\Theta$ and a $\cH^k$--rectifiable set $\Gamma$ such that for any Borel set $A$
  \begin{equation*}
    \mu(A) = \int_{A\cap\Gamma} \Theta \,\cH^k.
  \end{equation*}
\end{definition}

Since $\Gamma$ is $\cH^1$--rectifiable, at $\cH^1$--a.e.~point $x \in \Gamma$, it has a well-defined tangent space $T_x\Gamma$ and $\nu$ has a \defined{tangent measure}, i.e., the limit
\begin{equation*}
  T_x\nu := \lim_{\epsilon \to 0} \frac1\epsilon ({\exp}\circ s_\epsilon)^* \nu
\end{equation*}
exists and
\begin{equation*}
  T_x\nu = \Theta(x) \,\cH^1\lfloor T_x\Gamma.
\end{equation*}
Here $s_\epsilon(x) := \epsilon x$.

To prove \autoref{Prop_NuRectifiable} we will make use of the following deep theorem, whose proof is carefully explained in \cite{DeLellis2008}.

\begin{theorem}[\citet{Preiss1987}]
  \label{Thm_Preiss}
  If $\mu$ is a locally finite measure on $M$ and $m \in \N_0$ is such that for $\mu$--a.e. $x \in M$ the density
  \begin{equation*}
     \Theta_\mu^m(x) := \lim_{r \downarrow 0} \frac{\mu(B_r(x))}{r^m}.
  \end{equation*}
 exists and is finite, then $\mu$ is $\cH^m$--rectifiable.
\end{theorem}

\begin{proof}[Proof of \autoref{Prop_NuRectifiable}]
  The proof has five steps.

  \setcounter{step}{0}
  \begin{step}
    \label{Step_MuMonotonicity}
    With the same constant as in \autoref{Prop_Monotonicity} and for all $x \in M$ and $0 < s \leq r$
    \begin{equation*}
      e^{c s}s^{-1} \mu(B_s(x)) \leq e^{c r}r^{-1} \mu(B_r(x)) + cr^2.
    \end{equation*}
  \end{step}

  This is not quite a trivial consequence of \autoref{Prop_Monotonicity} because $(\mu_i)$ only weakly converges to $\mu$;
  hence, we only know that $\mu(\bar B_r(x)) \geq \limsup_{i\to\infty} \mu_i(\bar B_r(x))$ and $\liminf_{i\to\infty} \mu_i(B_r(x)) \geq \mu(B_r(x))$.

  For $x \in M$ set
  \begin{equation*}
    \sR_x := \set{ r \in (0,r_0] : \mu(\del B_r(x)) > 0 }.
  \end{equation*}
  If $r \notin \sR_x$, then it follows from \autoref{Prop_Monotonicity} that 
  \begin{equation*}
    e^{c s}s^{-1} \mu(B_s(x)) \leq e^{c r}r^{-1} \mu(B_r(x)) + cr^2.
  \end{equation*}
  The general case follows by an approximation argument.
  Note that $\sR_x$ is at most countable.
  Thus, given $r \in \sR_x$,
  we can find a sequence $(r_i)$ such that
  $s < r_i < r$, $r_i \notin \sR_x$, and $r := \lim_{i\to \infty} r_i$.
  By dominated convergence
  \begin{equation*}
    \mu(B_r(x)) = \lim_{i \to \infty} \mu(B_{r_i}(x)).
  \end{equation*}

  \begin{step}
    \label{Step_Theta}
    The limit
    \begin{equation*}
      \Theta(x) := \lim_{r\downarrow 0} r^{-1} \mu(B_r(x))
    \end{equation*}
    exists for all $x \in M$.
    The function $\Theta \co M \to [0,\infty)$ is upper semi-continuous, it vanishes outside $S$, is bounded and $\Theta(x) \geq \epsilon_0$ for all $x \in S$.
  \end{step}

  The existence of the limit is a direct consequence of \autoref{Step_MuMonotonicity}.

  To see that $\Theta$ is upper semi-continuous, let $(x_i)$ be a sequence of points in $M$ converging to a limit point $x = \lim_{i \to \infty} x_i$.
  Let $r \notin \sR_x$ and $\epsilon > 0$.
  For $i \gg 1$
  \begin{align*}
    \Theta(x_i)
    \leq e^{c r} r^{-1} \mu(B_r(x_i)) + c r^2
    \leq e^{c r} r^{-1} \mu(B_{r+\epsilon}(x)) + c r^2.
  \end{align*}
  Therefore, $\limsup_{i \to \infty} \Theta(x_i) \leq e^{c r} r^{-1} \mu(B_{r}(x)) + cr^2$.
  Taking the limit as $r \to 0$ shows that $\Theta$ is upper semi-continuous.

  The last part is clear.
  
  \begin{step}
    \label{Step_ThetaU}
    $\Theta_u^*$ vanishes $\cH^{1}$--a.e. in $M$, i.e., $\cH^1(\sing(u)) = 0$.
  \end{step}

  Given $\epsilon > 0$, set
  \begin{align*}
    E_\epsilon := \set{ x \in M : \Theta_u^*(x) > \epsilon }.
  \end{align*}
  Given $\delta > 0$, choose $\{x_1, \ldots, x_m\} \subset E_\epsilon$ and $\{r_1, \ldots, r_m\} \subset (0,\delta]$ such that the balls
  $B_{2r_j}(x_j)$ cover $E_\epsilon$, but the balls $B_{r_j}(x_j)$ are pairwise disjoint.
  Moreover, we can arrange that
  \begin{align*}
    \frac{1}{r_j} \int_{B_{r_j}(x_j)} \abs{\nabla u}^2 > \epsilon.
  \end{align*}
  Since $u$ is smooth on $M \setminus S$, we must have $E_\epsilon \subset S$.
  Hence,
  \begin{equation*}
    \sum_{j=1}^m r_j
    \leq \frac{1}{\epsilon} \sum_{j=1}^m \int_{B_{r_j}(x_j)} \abs{\nabla u}^2
    \leq \frac{1}{\epsilon} \int_{N_\delta(S)} \abs{\nabla u}^2
  \end{equation*}
  where $N_\delta(S) = \{ x \in M : d(x,S) < \delta \}$.
  The right-hand side goes to zero as $\delta$ goes to zero.
  Thus $\cH^1(E_\epsilon) = 0$ for all $\epsilon > 0$.
  This concludes the proof.

  \begin{step}
    $\nu$ is $\cH^1$--rectifiable.
  \end{step}

  By \autoref{Step_Theta} for any $x \in M \setminus \sing(u)$ the density
  \begin{equation*}
    \Theta_\nu(x) = \lim_{r \downarrow 0} \frac{\nu(B_r(x))}{r}
  \end{equation*}
  exists and agrees with $\Theta(x)$.
  In general $\Theta^*_\nu \leq \Theta < \infty$, which implies that $\nu \ll \cH^1$ (see, e.g., \cite[Proposition 2.2.2]{Krantz2008}).
  By \autoref{Step_ThetaU}, $\cH^1(\sing(u)) = 0$ and, hence, $\nu(\sing(u)) = 0$.
  Applying \autoref{Thm_Preiss} yields the assertion.

  \begin{step}
    We prove the proposition.
  \end{step}

  We have already proved the assertion about $\sing(u)$.
  Since $\nu$ is $\cH^1$--rectifiable and $\Gamma = \supp(\nu)$, it follows that $\Gamma$ is $\cH^1$--rectifiable and $\nu$ can be written as
  \begin{equation*}
    \nu = \tilde\Theta \,\cH^1\lfloor \Gamma
  \end{equation*}
  for some $\tilde \Theta$.
  By \autoref{Step_ThetaU}, $\Theta_\nu(x) = \tilde\Theta(x)$ for $\cH^1$--a.e.~$x \in \Gamma$.
\end{proof}


\section{Bubbling analysis}
\label{Sec_BubblingAnalysis}

We will now show that the ``lost energy'' goes into the formation of bubbles transverse to $\Gamma$.
To state the main result recall that an orientation on $N_x\Gamma$ induces a canonical complex structure and an orientation of $N_x\Gamma$ is canonically determined by the choice of a unit tangent vector $v \in T_x\Gamma \subset T_xM$ since $M$ is oriented.

\begin{prop}
  \label{Prop_Bubbles}
  If $x \in \Gamma$ is smooth, i.e., $T_x\Gamma$ exists and $x \notin \sing(u)$, then there exists a $(-I(v))$--holomorphic sphere $\fz_x \co N_x\Gamma \cup \set{\infty} \to X := \fX_x$ with
  \begin{equation}
    \label{Eq_EnergyInequality}
    \sE(\fz_x) := \int_{S^2} \abs{\rd \fz_x}^2 \leq \Theta(x).
  \end{equation}
  Here we have picked some unit vector $v \in T_x\Gamma$.
\end{prop}

\begin{remark}
  It is immaterial whether we choose $v$ or its opposite $-v$ since this results in changing the complex structures on both $N_x\Gamma$ and $X$.
  In particular, the above cannot be used to fix an orientation of $\Gamma$;
  however, the existence of $\fz_x$ does restrict the possible tangent directions, see \autoref{Sec_TangentDirections}.
\end{remark}

\begin{remark}
  The reason that \eqref{Eq_EnergyInequality} may be strict is that we only extract one bubble of what is an entire bubbling-tree, cf.~\citet{Parker1993} for the general notion of a bubbling tree, and \citet[Section 7]{Bellettini2015} for a discussion on how to extract a bubbling tree in the our situation.
\end{remark}

The holomorphic sphere $\fz_x$ is obtained by blowing-up $(u_i)$ around the point $x \in \Gamma$.
We assume a trivialisation of $\fX$ in a neighbourhood $U$ of $x$ has been fixed;
see \autoref{Ex_FueterMap}.
We use the following notation:
given any map $u \co U \to X$ and a scale factor $\lambda > 0$, we define a rescaled map $u_{\lambda} \co B^3_{r_0/\lambda}(0) \to X$ by
\begin{equation}
  \label{Eq_RescaledMap}
  u_\lambda := u({\exp}\circ s_\lambda).
\end{equation}
with $s_\lambda(y) := \lambda y$.
We write $(z,w)$ to denote points in $T_x\Gamma \times N_x\Gamma = T_xM$ and work with generalised cubes of the form
\begin{equation*}
  Q_{r,s}(z_0,w_0) := B_r(z_0) \times B_s(w_0) \subset T_x\Gamma \times N_x\Gamma = T_xM.
\end{equation*}

\begin{proof}[Proof of \autoref{Prop_Bubbles}]
  We proceed in four steps.

  \setcounter{step}{0}
  \begin{step}[Preliminary scale fixing]
    \label{Step_TangentMeasureScale}
    There exists a null-sequence $(\epsilon_i) \subset (0,1)$ such that
    \begin{equation*}
      \abs{\rd u_{i;\epsilon_i}}^2 \,\cH^3 \wto T_x\nu = \Theta(x) \,\cH^1\lfloor T_x\Gamma.
    \end{equation*}
  \end{step}

  By definition, $T_x\nu$ is the weak limit of $\epsilon^{-1}({\exp}\circ s_{\epsilon})^*\nu$ as $\epsilon$ tends to zero.
  Since $x \notin \sing(u)$, we have
  \begin{equation*}
    \lim_{\epsilon \to 0} \frac{1}{\epsilon}({\exp}\circ s_\epsilon)^*\nu
    = \lim_{\epsilon \to 0} \frac{1}{\epsilon}({\exp}\circ s_\epsilon)^*\mu.
  \end{equation*}
  Thus
  \begin{equation*}
    T_x\nu = \lim_{\epsilon \to 0}  \lim_{i \to \infty} \frac1\epsilon({\exp}\circ s_\epsilon)^*\mu_i
    =\lim_{i \to \infty} \frac1\epsilon({\exp}\circ s_{\epsilon_i})^*\mu_i
  \end{equation*}
  for some null-sequence $(\epsilon_i)$.
  This implies the assertion since
  \begin{equation*}
    \frac1{\epsilon_i}({\exp}\circ s_{\epsilon_i})^*\mu_i = \abs{\rd u_{i;\epsilon_i}}^2 \,\cH^3.
  \end{equation*}

  \begin{step}[Asymptotic translation invariance]
    \label{Step_TranslationInvariance}
    After passing to a subsequence, we can assume that there exists a null-sequence $(z_i)$ such that
    \begin{equation}
      \label{Eq_TranslationInvariance}
      \lim_{i\to\infty} \sup_{s \leq 1} \frac1s \int_{Q_{s,1}(z_i,0)} \abs{\del_v u_{i;\epsilon_i}}^2 = 0.
    \end{equation}
  \end{step}

  \setcounter{substep}{0}
  \begin{substep}
    We have
    \begin{equation*}
      \lim_{i\to \infty} \int_{Q_{2,1}(0)} \abs{\del_v u_{i,\epsilon_i}}^2 = 0.
    \end{equation*}
  \end{substep}

  Denote by $\del_\rho$ the radial vector field emanating from $4v$.
  By \autoref{Prop_Monotonicity}, for for $0 < s \leq r$
  \begin{multline}
    \label{Eq_PreTranslationInvariance}
    \int_{B_r(4v)\setminus B_s(4v)} e^{c \epsilon_i \tau}\tau^{-1} \abs{\del_\rho u_{i;\epsilon_i}}^2 \\
    \leq e^{c \epsilon_i r} r^{-1} \int_{B_r(4v)} \abs{\rd u_{i;\epsilon_i}}^2 - e^{c \epsilon_i s} s^{-1} \int_{B_s(4v)} \abs{\rd u_{i;\epsilon_i}}^2 + c\epsilon_i^2r^2.
  \end{multline}
  As $i$ tends to infinity the first two terms on the right-hand side both converge to $\Theta(x)$, since $T_x\nu = \Theta(x) \,\cH^1\lfloor T_x\Gamma$ and the last term tends to zero.

  Since $Q_{2,1}(0) \subset B_8(4v)\setminus B_1(4v)$, it follows that
  \begin{equation*}
    \lim_{i\to \infty} \int_{Q_{2,1}(0)} \abs{\del_\rho u_{i,\epsilon_i}}^2 = 0.
  \end{equation*}
  This completes the proof, because along $T_x\Gamma \cap B_2(0)$ the vector fields $\del_\rho$ and $v$ are colinear and $\abs{\del_v u_{i,\epsilon_i}}^2 \,\cH^3$ converges to zero outside $T_x\Gamma$.

  \begin{substep}
    \label{Step_AETranslationInvariance}
    For $\cH^1$--a.e.~$z \in B_1(0) \subset T_x\Gamma$
    \begin{equation}
      \label{Eq_AETranslationInvariance}
      \lim_{i\to\infty} \sup_{s \leq 1} \frac1s \int_{Q_{s,1}(z,0)} \abs{\del_v u_{i;\epsilon_i}}^2 = 0.
    \end{equation}
  \end{substep}

  Define $f_i\co B_2(0) \subset T_x\Gamma \to [0,\infty)$ by
  \begin{equation*}
    f_i(z) := \int_{B_1(0) \subset N_x\Gamma} \abs{\del_v u_{i;\epsilon_i}}^2(z,\cdot)
  \end{equation*}
  and denote by $Mf_i\co B_1(0) \subset T_x\Gamma \to [0,\infty)$ the Hardy--Littlewood maximal function associated with $f_i$:
  \begin{equation*}
    Mf_i(z) := \sup_{s \leq 1} \frac1s \int_{B_s(z) \subset T_x\Gamma} f_i.
  \end{equation*}
  We need to show that the set
  \begin{equation*}
    A = \set{ z \in B_1(0) : \liminf_{i\to\infty} Mf_i(z) > 0}
  \end{equation*}
  is such that $\cH^1(A) = 0$.
  If we set
  \begin{equation*}
    A_{i,\delta} :=  \set{ z \in B_1(0) : Mf_i(z) \geq \delta},
  \end{equation*}
  then
  \begin{equation*}
    A = \bigcup_{\delta > 0} \bigcup_{I \in \N} \bigcap_{i = I}^\infty A_{i,\delta}.
  \end{equation*}
  
  By the weak-type $L^1$ estimate for the maximal operator, for each $\delta > 0$
  \begin{equation*}
    \cH^1(A_{i,\delta}) \lesssim \frac{\Abs{f_i}_{L^1}}{\delta}.
  \end{equation*}
  Since $\Abs{f_i}_{L^1} \to 0$, we have
  \begin{equation*}
    \cH^1\(\bigcap_{i = I}^\infty A_{i,\delta}\) = 0;
  \end{equation*}
  hence, $\cH^1(A) = 0$ by monotonote convergence.

  \begin{substep}
    We prove \eqref{Eq_TranslationInvariance}.
  \end{substep}

  By \autoref{Step_AETranslationInvariance}, for each $j \in \N$ we can find $z_j \in B_{1/j}(0)$ such that
  \begin{equation*}
    \lim_{i\to\infty} \sup_{s \leq 1} \frac1s \int_{Q_{s,1}(z_j,0)} \abs{\del_v u_{i;\epsilon_i}}^2 = 0.
  \end{equation*}
  Now apply a diagonal sequence argument.

  \begin{step}[Bubble detection]
    \label{Step_BubbleDetection}
    There exists a null-sequence $(\delta_i) \in (0,1/2)$ such that, for each $i \gg 1$,
    \begin{equation}
      \label{Eq_BubbleDetection}
      \max_{w \in \bar B_{1/2}(0)} \frac{1}{\delta_i} \int_{B_{\delta_i}(z_i, w)} \abs{\rd u_{i,\epsilon_i}}^2 = \epsilon_0/8;
    \end{equation}
    moreover, if $w_i \in \bar B_{1/2}(0)$ denotes a point at which this maximum is already, then $(w_i)$ is a null-sequence.
  \end{step}

  By \autoref{Step_TangentMeasureScale}, we have
  \begin{equation*}
    \liminf_{i\to\infty} \max_{w \in \bar B_{1/2}(0)} \frac{1}{\delta} \int_{B_{\delta}(z_i,w)} \abs{\rd u_{i;\epsilon_i}}^2 = \Theta(x) \geq \epsilon_0
  \end{equation*}
  for all $\delta > 0$,
  while for fixed $i \in \N$ and $w \in \bar B_{1/2}(0) \subset N_x\Gamma$
  \begin{equation*}
    \lim_{\delta \downarrow 0} \frac{1}{\delta} \int_{B_{\delta}(z_i,w)} \abs{\rd u_{i;\epsilon_i}}^2 = 0.
  \end{equation*}
  Hence, we can find a null sequence $(\delta_i)$ such that
  \begin{equation*}
    \max_{w \in \bar B_{1/2}(0)} \frac{1}{\delta_i} \int_{B_{\delta_i}(z_i, w)} \abs{\rd u_{i,\epsilon_i}}^2 = \epsilon_0/8.
  \end{equation*}

  If (after passing to a subsequence) we can find $\sigma >0$ and $(w_i) \in \bar B_{1/2}(0)\setminus B_{\sigma}(0)$ such that the maximum in \eqref{Eq_BubbleDetection} is achieved at $w = w_i$, then by \autoref{Prop_Monotonicity} the density of $T_x\nu$ at $(0,w)$ would be positive, contradicting \autoref{Step_TangentMeasureScale}.

  \begin{step}
    We prove \autoref{Prop_Bubbles}.
  \end{step}

  Let $(w_i)$ be as in \autoref{Step_BubbleDetection}. 
  Define
  \begin{equation*}
    \tilde u_i := \tilde u_i(\cdot) := u_{i;\delta_i\epsilon_i}\bigl(\delta_i^{-1}(z_i,w_i)+\cdot\bigr).
  \end{equation*}
  By construction
  \begin{equation*}
    \max_{w \in B_{(1/2 - \abs{w_i})\delta_i}(0)} \int_{B_{1}(0, w)} \abs{\rd \tilde u_i}^2 = \epsilon_0/8
  \end{equation*}
  with the maximum achieved at $w = 0$.

  From \autoref{Prop_EpsilonRegularity} and \autoref{Rmk_HigherEpsilonRegularity} we obtain $C^\infty_\loc$--bounds on $\tilde u_i$ which allow us to pass to a limit $u \co B_1(0)\times N_x\Gamma \to X$, which solves the Fueter equation.
  It follows from \autoref{Step_TranslationInvariance}, that
  \begin{equation*}
    \lim_{i\to \infty} \int_{Q_{1,1/2\delta_i}(0,0)} \abs{\del_\nu \tilde u_i}^2 = 0.
  \end{equation*}
  Hemce, $u$ is going to be constant in $z \in B_1(0) \subset T_x\Gamma$;
  hence, $u$ is the pullback of a map $\fz\co N_x\Gamma \to X$.
  We can choose the orthonormal frame $(v_1,v_2,v_3)$ on $T_xM$ constant and with $v_1 = v \in T_x\Gamma$ and $v_2, v_3 \in N_x\Gamma$.
  With respect to this frame the Fueter operator takes the form
  \begin{equation*}
    \fF = I(v_1)\del_v + I(v_2)\delbar
  \end{equation*}
  with $\delbar = \del_{v_2} + (-I(v))\del_{v_3}$.
  Thus $\fz$ is $(-I(v))$--holomorphic.
\end{proof}

\begin{question}
  What happens near non-smooth points of $\Gamma$?
\end{question}


\section{Constraints on tangent directions}
\label{Sec_TangentDirections}

By \autoref{Prop_Bubbles}, if $x \notin \sing(u)$ and $v \in ST_x\Gamma$, then $\fX_x$ must admit a non-trivial $(-I(v))$--holomorphic sphere $\fz_x$ of area at most $\Theta(x)$.
Since $\Theta$ is upper semi-continuous, it achieves a maximum $A_{\max}$ on $\Gamma$.
Thus, the area of $\fz_x$ is bounded by $A_{\max}$ and the following shows that the possible tangent directions of $\Gamma$ are strongly constrained.

\begin{prop}
  \label{Prop_TangentDirections}
  Let $X$ be a simple hyperkähler manifold with $b_2(X) \geq 6$.
  Given $A_{\max} > 0$, there exists only finitely many $I_\xi \in \fH(X)$ for which there exists a rational curve $C$ in $(X,I_\xi)$ with
  \begin{equation*}
    \area(C) = \<[C],\omega_\xi\> \leq A_{\max}.
  \end{equation*}
  Here $\omega_\xi = g(I_\xi\cdot,\cdot)$.
\end{prop}

If $X$ is a $K3$ surface, then this is essentially contained in \citet[Proposition 3.1]{Bryan2000}.
Its proof mainly uses some facts about the $K3$--lattice $(H^2(K3,\Z),\cup)$.
The appropriate replacement of the cup-product for general simple hyperkähler manifold is the \defined{Beauville--Bogomolov--Fujiki (BBF) form} $q \co S^2H^2(X,\Z) \to \Z$.
We refer the reader to \cites{Beauville1983,Bogomolov1978,Fujiki1987} for details about the BBF form.
For our purposes it suffices to recall that:
\begin{itemize}
\item
  $q$ is non-degenerate, i.e., the induced map $H^2(X,\Q) \to H^2(X,\Q)^*$ is an isomorphism.
  In particular, for each $C \in H_2(X,\Z)$ there exists a unique $\gamma \in H^2(X,\Q)$ such that
  \begin{equation}
    \label{Eq_HolomorphicSphereDuality}
    q(\gamma,\cdot) = \<C,\cdot\> \in H^2(X,\Q)^*.
  \end{equation}
\item
  $q$ has signature $(3,b_2(X)-3)$ with ${\rm span}\set{ [\omega_\xi] : \xi \in S^2 }$ forming a maximal positive definite subspace.
  We denote the perpendicular maximal negative definite subspace by $N$.
\end{itemize}

\begin{theorem}[Amerik--Verbitsky]
  \label{Thm_BoundedBBFSquare}
  If $X$ is a simple hyperkähler manifold with $b_2(M) \geq 6$, then there exists an positive integer $\sigma \in \N$ such that
  \begin{equation*}
    q(\gamma,\gamma) \geq -\sigma
  \end{equation*}
  for all $\gamma \in H^2(X,\Q)$ with \eqref{Eq_HolomorphicSphereDuality} for some $C$ represented by a $I_\xi$--holomorphic sphere for some $I_\xi \in \fH(X)$.
\end{theorem}

\begin{proof}
  This follows by observing that $\gamma$ is a MBM class in the sense of \cite[Definition 2.14]{Amerik2014} and then appealing to \cite[Theorem 5.3]{Amerik2014}.
\end{proof}

\begin{remark}
  \autoref{Thm_BoundedBBFSquare} generalises the fact that any class representing a holomorphic sphere in $K3$ has square $-2$.
\end{remark}

\begin{prop}
  \label{Prop_GammaConstraints}
  There exists a constant $c_0 > 0$ such that if $C$ is represented by a $I_\xi$--holomorphic sphere of area $A$, then $\gamma$ as in \eqref{Eq_HolomorphicSphereDuality} is of the form
  \begin{equation}
    \label{Eq_RationalClassDecomposition}
    \gamma = \beta + c_0 A \omega_\xi 
  \end{equation}
  with $\beta \in N$ and
  \begin{equation*}
    q(\beta,\beta) \geq -\sigma - c_0 A^2.
  \end{equation*}
\end{prop}

\begin{proof}
  It follows from \eqref{Eq_HolomorphicSphereDuality} that
  \begin{equation}
    q(\gamma,\omega_\eta) = 0
  \end{equation}
  for all $\eta \perp \xi$;
  hence, $\gamma = \beta + c_0A \omega_\xi$ with $c_0 = 1/q(\omega_\xi,\omega_\xi)$, which does not depend on $\xi \in S^2$, and $\beta \in N$.
  Since $q(\gamma,\gamma) \geq -\sigma$, we have
  \begin{equation*}
    q(\beta,\beta) \geq -\sigma - c_0 A^2.
    \qedhere
  \end{equation*}
\end{proof}

\begin{proof}[Proof of \autoref{Prop_TangentDirections}]
  There are only finitely many $\gamma$ as in  \autoref{Prop_GammaConstraints} with $A \leq A_{\max}$ and $\gamma$ determines $\xi \in S^2$ uniquely.
\end{proof}


\appendix
\section{The Heinz trick}
\label{A_HeinzTrick}

Throughout we consider a bounded open subset $U \subset \R^n$ endowed with a smooth metric $g$ which extends smoothly to $\bar U$.
Implicit constants are allowed to depend on the geometry of $U$.

\begin{lemma}[\citet{Heinz1955}]
  \label{Lem_HeinzTrick}
  Fix $d > 0$ and set
  \begin{equation*}
    q := \frac{2}{d} + 1.
  \end{equation*}
  Suppose $f\co U \to [0,\infty)$ and $p,\delta \in \set{0, 1}$ are such that the following hold:
  \begin{enumerate}
  \item
    \label{I_DifferentialInequality}
    We have
    \begin{equation*}
      \Delta f \lesssim f^q +f^p.
    \end{equation*}
  \item
    \label{I_Monotonicity}
    If $B_s(y) \subset B_{r/2}(x) \subset U$, then
    \begin{equation*}
      s^{d-n}\int_{B_s(y)} f \lesssim r^{d-n}\int_{B_r(x)} f + \delta r^2.
    \end{equation*}
  \end{enumerate}
  Then there exists a constant $\epsilon_0>0$ such that for all $B_r(x) \subset U$ with
  \begin{align*}
    \epsilon = r^{d-n}\int_{B_r(x)}f \leq \epsilon_0
  \end{align*}
  we have
  \begin{equation*}
    \sup_{y \in B_{r/4}(x)} f(y) \lesssim r^{-d} \epsilon + \((1-p) + \delta\)r^2.
  \end{equation*}
\end{lemma}

\begin{remark}[Heinz trick in the subcritical case]
  \label{Rmk_SubcriticalHeinzTrick}
  If $n < d$,
  \begin{equation*}
    \epsilon \leq \epsilon_0 \quad\text{whenever}\quad
    r \leq \(\frac{\epsilon_0}{\int_U f}\)^{\frac{1}{d-n}}.
  \end{equation*}
  In particular, for all compact $K \subset U$, $\Abs{f}_{L^\infty(K)}$ is bounded a priori depending only on $\int_U f$ and $d(K,\del U)$.
\end{remark}

We use the following standard result; see \cite[Theorem 9.20]{Gilbarg2001} or \cite[Proof of Theorem B.1]{Hohloch2009}.

\begin{prop}
  \label{Prop_MeanValue}
  For all $B_r(x) \subset U$ and every smooth function $f \co B_r(x) \to [0,\infty)$
  \begin{equation*}
    f(x) \lesssim r^{-n} \int_{B_r(x)} f~\vol + r^2\Abs{\Delta f}_{L^\infty}.
  \end{equation*}
\end{prop} 

\begin{proof}[Proof of Lemma \ref{Lem_HeinzTrick}]
  Define a function $\theta \co B_{r/2}(x) \to [0,\infty)$ by
  \begin{equation*}
    \theta(y) := \left(\frac{r}{2} - d(x,y)\right)^{d} f(y).
  \end{equation*}
  Since $\theta$ is non-negative and vanishes on the boundary of $B_{\frac r2}(x)$, it achieves its maximum
  \begin{equation*}
    M := \max_{y\in B_{\frac r2}(x)} \theta(y)
  \end{equation*}
  in the interior of $B_{\frac{r}{2}}(x)$.
  We will derive a bound for $M$, from which the assertion follows at once.

  Let $y_0$ be a point with $\theta(y_0) = M$, set
  \begin{equation*}
    F := f(y_0)
  \end{equation*}
  and denote by
  \begin{equation*}
    s_0 := \frac{1}{2} \left(\frac r2-d(x,y_0)\right)
  \end{equation*}
  half the distance from $y_0$ to the boundary of $B_{\frac r2}(x)$.
  Each $y \in B_{s_0}(y_0)$ has distance from the boundary of $B_{\frac r2}(x)$ at least $s_0$;
  hence,
  \begin{align*}
    f(y) \leq s_0^{-d}\theta(y) \leq s_0^{-d}\theta(y_0) \lesssim F.
  \end{align*}
  \autoref{Prop_MeanValue} applied to $B_s(y_0)$ together with (\ref{I_DifferentialInequality}) and the above bound yields
  \begin{align*}
    F \lesssim s^{-n} \int_{B_s(y_0)} f + s^2 \left(F^q + F^p\right)
  \end{align*}
  for all $0 \leq s \leq s_0$.
  Combined with (\ref{I_Monotonicity}) this becomes
  \begin{align*}
    F \lesssim s^{-d} \epsilon + s^2 \(F^q + F^p\) +\delta r^2,
  \end{align*}
  which can be rewritten as
  \begin{equation}
    \label{Eq_PolynomialInequality}
    s^dF \lesssim \epsilon + s^{d+2}\(F^q + F^p\) + \delta r^2s^d.
  \end{equation}
  This inequality will yield the desired bound on $M$.
  It is useful to make a case distinction.

  \setcounter{case}{0}
  \begin{case}
    $F \leq 1$.
  \end{case}

  In this case a bound on $M$ follows from simple algebraic manipulations.
  If $p=0$ or $\delta = 1$, then \eqref{Eq_PolynomialInequality} with $s = s_0$ yields
  \begin{equation*}
    M = \theta(y_0) \lesssim s_0^dF \lesssim \epsilon + r^{d+2}.
  \end{equation*}
  If $p=1$ and $\delta = 0$, this bound can be sharpened.
  \eqref{Eq_PolynomialInequality} becomes
  \begin{equation*}
    s^dF \leq \frac{c\epsilon}{1-c s^2}.
  \end{equation*}
  If $cs_0^2 \leq \frac12$, then we obtain
  \begin{align*}
    M \lesssim s_0^dF \lesssim \epsilon;
  \end{align*}
  otherwise, setting $s:=(2c)^{-\frac12} \leq s_0$ yields
  \begin{align*}
    F \lesssim\epsilon,
  \end{align*}
  and thus $M \lesssim \epsilon.$

  \begin{case}
    $F > 1$.
  \end{case}
  
  From \eqref{Eq_PolynomialInequality} we derive
  \begin{align*}
    s^d F \lesssim \epsilon + s^{d+2} F^q + \delta r^2 s^d
  \end{align*}
  for all $0 \leq s \leq s_0$.
  Set $t := t(s) = sF^{1/d}$.
  Then the above inequality can be expressed as
  \begin{align*}
    t^d(1-ct^2) \leq c(\epsilon + \delta r^2).
  \end{align*}
  For sufficiently small $\epsilon > 0$, the corresponding equation $t^d(1 - c t^2) = c (\epsilon+\delta r^2)$ has $d$ small roots $t_1,\ldots,t_d$, which are approximately $\pm (c\epsilon+c\delta r^2)^{\frac1d}$, and two large roots.
  Since $t(0) = 0$ and by continuity, for each $s \in [0,s_0]$, $t(s)$ must be less than the smallest positive root;
  hence, $t(s) \lesssim (\epsilon+\delta r^2)^{\frac1d}$ for all $s \in [0,s_0]$.
  This finishes the proof.
\end{proof}


\section{Compactness for Fueter maps with four dimensional source manifold}
\label{Sec_DimensionFour}

\begin{prop}
  Let $V$ be a $4$--dimensional Euclidean vector space, $H$ a quaternionic vector space, $I \co S\Lambda^+V^* \to S(\Im\H)$ an isometric identification of the unit length self-dual forms on $V$ with the unit imaginary quaternions and $\iota \co \Lambda^+V^* \to \so(V)$.
  The endomorphism $\Psi \in \End(\Hom(V,H))$ defined by
  \begin{equation*}
    \Psi T := \sum_{i=1}^3 I(\omega_i)\circ T \circ \iota(\omega_i)
  \end{equation*}
  has eigenvalues $1$ and $-3$.
  Here we sum over an orthonormal basis $(\omega_1,\omega_2,\omega_3)$ of $\Lambda^+V^*$.
  We denote the $(-3)$--eigenspace by $\Hom_I(V,H)$.
\end{prop}

Let $M$ be an orientable Riemannian $4$--manifold, let $\fX \xrightarrow{\pi} M$ be a bundle of hyperkähler manifolds together with a fixed identification $I \co S\Lambda^+T^*M \to \fH(\fX)$ of the unit sphere bundle of self-dual forms on $M$ and the bundle of hyperkähler spheres of the fibres of $\fX$ and fix a connection on $\fX$.

\begin{definition}
  A section $u \in \Gamma(\fX)$ is called a \emph{Fueter section} if
  \begin{equation}
    \label{Eq_4DFueter}
    \fF u := \nabla u - \Psi\nabla u = 0 \in \Gamma(u^*\Hom_I(\pi^*TM,V\fX)).
  \end{equation}
\end{definition}

\begin{remark}
  If $M = \R \times N$ for some $3$--manifold $N$, $\fX$ is the pullback of a bundle $\fY$ of hyperkähler manifolds on $N$, $I$ is obtained from an identification $J\co STM \iso \fH(\fX)$ and the connection on $\fX$ is the pullback of a connection on $\fY$, then \eqref{Eq_4DFueter} can be written as
  \begin{equation*}
    \del_t u - \fF u = 0
  \end{equation*}
  with $\fF$ denoting the $3$--dimensional Fueter operator.
  This is the form in which the $4$--dimensional Fueter operator appears in \cite{Hohloch2009}.
\end{remark}

\begin{remark}
  Unlike in the $3$--dimensional case, $\Lambda^+T^*M$ need not be trivial.\footnote{$\Lambda^+T^*M$ being trivial is equivalent to $3\sigma(M)+2\chi(M)=0$ and $w_2(M) = 0$.}
  Thus the analogue of the setup in \autoref{Ex_FueterMap} rarely makes sense globally, and one is almost forced to work with bundles of hyperkähler manifolds.
\end{remark}

The analogue of \autoref{Thm_A} in the $4$--dimensional case is the following result.

\begin{theorem}
  \label{Thm_B}
  Suppose $\fX$ is compact.
  Let $(u_i)$ be a sequence of solutions of the (perturbed) Fueter equation
  \begin{equation*}
    \fF u_i = \fp \circ u_i
  \end{equation*}
  with $\fp \in \Gamma(\fX,\Hom_I(\pi^*TM,V\fX))$ and
  \begin{equation*}
    \sE(u_i) := \int_{M} |\nabla u_i|^2 \leq c_\sE
  \end{equation*}
  for some constant $c_\sE > 0$.
  Then (after passing to a subsequence) the following holds:
  \begin{itemize}
  \item
    There exists a closed subset $S$ with $\cH^2(S) < \infty$ and a Fueter section $u \in \Gamma(M \setminus S, \fX)$ such that $u_i|_{M\setminus S}$ converges to $u$ in $C^\infty_\loc$.
  \item 
    There exist a constant $\epsilon_0 > 0$ and an upper semi-continuous function $\Theta \co S \to [\epsilon_0,\infty)$ such that the sequence of measures $\mu_i := |\nabla u_i|^2 \,\cH^4$ converges weakly to $\mu = |\nabla u|^2 \,\cH^4 + \Theta \,\cH^2\lfloor S$.
  \item
    $S$ decomposes as
    \begin{equation*}
      S = \Gamma \cup \sing(u)
    \end{equation*}
    with
    \begin{align*}
      \Gamma
      &:=
        \supp(\Theta \,\cH^1\lfloor S) \qand \\
      \sing(u)
      &:=
        \set*{
          x \in M
          :
          \limsup_{r\downarrow 0} \frac1{r^2} \int_{B_r(x)} |\nabla u|^2 > 0
        }.
    \end{align*}
    $\Gamma$ is $\cH^2$--rectifiable, and $\cH^2(\sing(u)) = 0$.
  \item
    For each smooth point of $\Gamma$ there exists a non-trivial holomorphic sphere in $\fz_x \co S^2 \to (\fX_x,-I(\xi))$ with $\xi$ a unit self-dual $2$--form on $T_xM$, whose associated complex structure preserves the splitting  $T_xM = T_x\Gamma \oplus N_x\Gamma$.
    Moreover,
    \begin{equation*}
      \Theta(x) \geq \sE(\fz_x) := \int_{S^2} |\rd \fz_x|^2.
    \end{equation*}
  \item
    If $\fX$ is a bundle of simple hyperkähler manifolds with $b_2 \geq 6$, then there is a subbundle $\fri \subset \set{ I \in \End(TM) \co I^2 = -\id }$, depending only on $\sup \Theta$, whose fibres are finite sets such that $T_x\Gamma$ is complex with respect to a complex structure $I \in \fri_x$ for all smooth points $x \in \Gamma$.
  \end{itemize}
\end{theorem}

\begin{proof}[Sketch of the proof]
  The proof is analogous to that of \autoref{Thm_A} with a few minor modifications:
  \begin{itemize}
  \item
    The renormalised energy is now
    \begin{equation*}
      \frac{1}{r^2} \int_{B_r(x)} |\nabla u|^2.
    \end{equation*}
  \item 
    In the proof of the monotonicity formula one now uses the $4$--form $\Lambda \in \Omega^4(\fX)$ obtained from the section of $\Lambda^+\pi^*TM \otimes \Lambda^2V\fX$ induced by $I$.
    Direct computation shows that \eqref{Eq_Energy} still holds.
    Similarly, one can verify the analogue of \eqref{Eq_BoundaryContribution}.
  \item
    The proof of the $\epsilon$--regularity and convergence outside $S$ carry over mutatis mutandis.
  \item
    In the bubbling analysis, $u_{i;\lambda_i}$ will be asymptotically translation invariant in the direction of $T_x\Gamma$.
    Fix a unit vector $v_0 \in T_x\Gamma$.
    Since, asymptotically, everything is invariant in the direction of $v_0$, we arrive back at the situation in \autoref{Sec_BubblingAnalysis}.
    \qedhere
  \end{itemize}
\end{proof}


\paragraph{Acknowledgements}
I am grateful to Misha Verbitsky for pointing out his work with Amerik \cite{Amerik2014}, to Costante Bellettini for a discussion of \cite{Bellettini2015} and to Gregor Noetzel and for insightful comments.
I also thank the referee for carefully reading this paper and making numerous helpful comments and suggestions.

\printreferences

\end{document}
